\documentclass[11pt]{article}
\usepackage{latexsym}
\usepackage{amsmath,amsthm,amsfonts,amssymb}
\usepackage[T1]{fontenc}
\usepackage[latin1]{inputenc}
\usepackage{aeguill}% pour les guillemets francais et autres bizrreries du latin1
\usepackage{url}%utile pour les preprints avec url
\usepackage[shortlabels]{enumitem}% enums intelligents
\usepackage{graphicx}
\usepackage[a4paper,textwidth=15cm,textheight=25cm]{geometry}
\usepackage{xspace}
\usepackage{export}%pour sauvecompteurs
\newtheorem{thm}{Theorem}[section]

\newtheorem*{thm*}{Theorem}
\newtheorem{dfn}[thm]{Definition} 
\newtheorem*{dfn*}{Definition}

\newtheorem{cor}[thm]{Corollary}
\newtheorem*{cor*}{Corollary}

\newtheorem{prop}[thm]{Proposition} 
\newtheorem*{prop*}{Proposition} 
\newtheorem*{properties*}{Properties} 
 
\newtheorem{lem}[thm]{Lemma} 
\newtheorem*{lem*}{Lemma}

\newtheorem*{claim*}{Claim} 
 
\newtheorem*{fact*}{Fact}

\newtheorem*{qst*}{Question}

\newtheorem*{pb*}{Problem}

\theoremstyle{remark}
 
\newtheorem*{algo*}{Algorithm} 
\newtheorem*{rem*}{Remark}
\newtheorem{rem}[thm]{Remark}
\newtheorem*{example*}{Example}
\newtheorem{example}[thm]{Example}

\setlist{topsep=0mm,itemsep=1mm,parsep=1mm}

\newcounter{numEnonceTmpInterne}% ce compteur sert a avoir un nom d'environnement unique dans enonce*
\newenvironment{enonce*}[1]{\theoremstyle{plain}\stepcounter{numEnonceTmpInterne}%
\def\a{enoncetmp\alph{numEnonceTmpInterne}}%
\newtheorem*{\a}{#1}\begin{\a}}{\end{\a}}

\makeatletter
\edef\@tempa#1#2{\def#1{\mathaccent\string"\noexpand\accentclass@#2 }}
\@tempa\rond{017}
\makeatother

\newcommand{\es}{\emptyset}
\renewcommand{\phi}{\varphi} 
\newcommand{\m} {^{-1}} 
\newcommand{\eps} {\varepsilon}

\newcommand {\ra} {\rightarrow}

\newcommand{\ie} {i.e.\ }

\newcommand {\cala} {{\mathcal {A}}}

\newcommand {\cald} {{\mathcal {D}}}   
   
\newcommand {\calf} {{\mathcal {F}}}   
\newcommand {\calg} {{\mathcal {G}}}   
\newcommand {\calh} {{\mathcal {H}}}

\newcommand {\bbZ} {{\mathbb {Z}}}   

\newcommand{\grp}[1]{\langle #1 \rangle}

\newcommand{\Out} {{\mathrm{Out}}}
\newcommand{\Inn} {{\mathrm{Inn}}}

\newcommand{\Aut} {{\mathrm{Aut}}}

\newcommand{\ad} {\mathrm{ad}}

\usepackage{srcltx}
\usepackage[all]{xypic}

\newcommand{\Tcan}{T_{\mathrm{can}}}
\newcommand{\Gcan}{\Gamma_{\mathrm{can}}}
\newcommand{\Z}{{\mathbb {Z}}}

\usepackage{pdfsync}
\newcommand{\inc}{\subset}

\newcommand {\N} {{\mathbb {N}}}

\newcommand{\fini}{{\calf in}}

\begin{document}

\title{Vertex finiteness for splittings of relatively hyperbolic groups}
\author{Vincent Guirardel and Gilbert Levitt}
%\date{\today.\\ \small Fichier \texttt{\jobname.tex}}

\maketitle

\begin{abstract}
Consider a group $G$ and a family $\cala$ of subgroups of $G$. We say that vertex finiteness holds for splittings of $G$ over $\cala$ if, up to isomorphism, there are only finitely many possibilities for vertex stabilizers of minimal $G$-trees with   edge stabilizers   in $\cala$. 

We show vertex finiteness when $G$ is a toral relatively hyperbolic group and $\cala$ is the family of abelian subgroups. 

We also show vertex finiteness when $G$ is hyperbolic relative to virtually polycyclic subgroups and $\cala$ is the family of virtually cyclic subgroups; if moreover $G$ is one-ended, there are only finitely many minimal $G$-trees with virtually cyclic  edge stabilizers, up to automorphisms of $G$.

\end{abstract}
 
 \section{Introduction} 

There are many results bounding the complexity of simplicial group actions on trees, or equivalently of  graph of groups decompositions. 
They go under the generic name of accessibility, and they are due mainly to Linnell, Dunwoody, Bestvina-Feighn, Sela, Weidmann  \cite{Linnell,Dun_accessibility,BF_bounding,Sela_acylindrical,Weidmann_accessibility}.
They play a key role in geometric group theory, for instance in the construction of JSJ decompositions or Makanin-Razborov diagrams. 

Accessibility usually provides bounds for the number of edges of graph of groups decompositions (splittings) of a given group $G$ over a certain family $\cala$ of edge groups  (hierarchical accessibility \cite{DePo_accessibilite,LoTo_accessibility} is   different).   
In this paper we are concerned with controlling the isomorphism type of vertex groups.

\begin{dfn}
Let $G$ be a group, and let $\cala$ be a family of subgroups closed under conjugating  and taking subgroups. We say that \emph{vertex finiteness} holds for splittings of $G$ over $\cala$ if, up to isomorphism, there are only finitely many possibilities for vertex groups of decompositions of $G$ as the fundamental group of a  
minimal graph of groups $\Gamma$ whose edge groups belong to $\cala$. 
\end{dfn}

 A graph of groups is \emph{minimal} if its Bass-Serre tree   is minimal, i.e.\ contains no proper $G$-invariant subtree. In the case of one-edge splittings, an HNN extension is always minimal; an amalgam $A*_CB$ is minimal if and only if $C\ne A,B$.   We always assume that $G$ is finitely generated, so minimal graphs of groups are finite.
 
   Equivalently, vertex finiteness states that there are finitely many isomorphism types for vertex stabilizers of minimal $G$-trees with edge stabilizers in $\cala$. 

Here are  standard examples of vertex finiteness:

\begin{itemize}
\item  $G$ is finitely generated, and $\cala$ only contains the trivial group.  
Vertex groups are free factors, there are only finitely many of them up to isomorphism. 

 \item $G$ is a free group $F_n$, and $\cala$ is the family of cyclic subgroups (splittings over $\cala$ are then called \emph{cyclic splittings}). 
Every vertex group of a cyclic splitting is  free   of  rank at most $n$ (this may be seen by abelianizing). 

\item  $G$ is the fundamental group of a closed orientable surface of genus $g$, 
and $\cala$ is the family of cyclic subgroups. Vertex groups are  
fundamental groups of embedded subsurfaces; they are   
free of  rank $\le 2g-1$.  

More generally, if $G$ is a one-ended hyperbolic group, and $\cala$ is the class of virtually cyclic groups,
  there are only finitely many possible  vertex groups up to the action of $\Aut(G)$    \cite{Sela_acylindrical,Delzant_accessibilite}.  
\end{itemize}

On the other hand, here are examples where finiteness does not hold,
 even if one restricts to  
amalgams or HNN extensions  (one-edge splittings): 

\begin{enumerate}[(i)]
 \item  Let $G=H*\Z$, with   $H$ containing torsion elements of arbitrarily high order $n$. Let $\cala$ be the class of finite groups. Then $ \Z/n\Z*\Z$ appears as a vertex group in the amalgam $G=H*_{\bbZ/n\bbZ} \left(\Z/n\Z *\Z\right)$.
There are examples with $G$ finitely presented (hence accessible). 
 
 \item Let $G$ be the Baumslag-Solitar group $BS(2,4)=\langle a,t\mid ta^2t\m=a^4\rangle$. For any $n\ge 1$, the group $ \langle x,y\mid x^{2^n} =y^2\rangle$ is a vertex group of a cyclic splitting of $G$ (see the introduction of \cite{Lev_GBS}).

\item In this example, $G$ is hyperbolic relative to the solvable subgroup  $BS(1,2)=\grp{a,t\mid tat\m=a^2}$,
$\cala$ is the class of cyclic groups, and there is no vertex finiteness even among $2$-acylindrical cyclic splittings.
Let $G=BS(1,2)*_{a=[x,y]} F(x,y)$, with $F(x,y) 
$ the free group on $x$ and $y$.
For each $n$, the element $a_n=t^{-n}at^n$ is a $2^n$-th root of $a$,
and  $P_n=\langle a_n,x,y\rangle\simeq \langle a_n,x,y \mid {a_n}^{2^n}=[x,y]\rangle$ 
is a vertex group of the   
cyclic splitting  $G=BS(1,2)*_{\grp{a_n}} P_n$.

\item  
  Let $H$ be the discrete Heisenberg group $H=\langle a,b,c\mid [a,b]=c, [a,c]=[b,c]=1\rangle$. Then $G=H*\Z$ is hyperbolic relative to the nilpotent group $H$,  
and there is no vertex finiteness among $2$-acylindrical splittings of $G$
over the class $\cala$ of nilpotent subgroups.
Indeed, $H$ has infinitely many non-isomorphic subgroups $H_n= \langle a^n,b^n,c \rangle$: they are distinguished by the index of the  derived subgroup in the center (we thank Pierre Pansu for suggesting this example). 
  Each group $ H_n*\Z$ is a vertex group in the 
splitting $G=H*_{H_n} (H_n*\Z)$.
\end{enumerate}

Our main result is the following:

\begin{thm} \label{main}
 Vertex finiteness holds in the following cases:
 
\begin{enumerate}
 \item $G$ is  finitely generated, $k$ is an integer, and $\cala=\fini_k$ is the family  of finite subgroups of order $\le k$;
 \item $G$ is hyperbolic relative to virtually polycyclic subgroups, and $\cala$ is the family of virtually cyclic (finite or infinite) subgroups;
 \item $G$ is hyperbolic relative to finitely generated abelian subgroups  (possibly with torsion), and $\cala$ is the family of   
 virtually  
abelian subgroups;
 \item $G$ is a finitely generated, torsion-free, CSA group, abelian subgroups of $G$ are finitely generated of bounded rank, and $\cala$ is the family of abelian subgroups. 
\end{enumerate}
\end{thm}

In   Assertion 3, groups in $\cala$ are abelian or virtually cyclic. A group   
is \emph{CSA} if maximal abelian subgroups are malnormal.

Note that   Assertion  2 (or   3) implies  vertex finiteness for splittings of hyperbolic groups  (with an arbitrary number of ends) over virtually cyclic subgroups.   Assertion  3 applies to abelian splittings (i.e.\ splittings over abelian groups) of limit groups, since by   \cite{Ali_combination, Dah_combination} limit groups are   toral relatively hyperbolic (i.e.\ torsion-free and hyperbolic relative to finitely generated abelian groups).  

\begin{rem}[Optimality]
Example  (i)  above shows  that bounding the order of edge groups is necessary in  Assertion 1, even if $G$ is finitely presented. 
 Assertion 2 does not apply to groups  which are hyperbolic relative to solvable groups, by Example (iii),
and acylindricity does not help.

  The example in  Subsection \ref{sec_heis}  will show  that Assertion 3 does not extend
if nilpotent parabolic subgroups are allowed.
We do not know whether  virtually abelian parabolic groups may be allowed (see \cite{GL_extension}  for the case of groups having a finite index subgroup as in Assertion  3).  Finally,   bounding the rank of abelian subgroups is necessary in   4:  
if $H$ contains $\Z^{n }$, then $\Z^{n }*\Z$ is a vertex group in a  splitting of $H*\Z$ over $\Z^n$.
\end{rem}

The following property, which we call \emph{tree finiteness},
is   stronger   than vertex finiteness: 
 there are only finitely many minimal splittings of $G$ over $\cala$, up to the action of $\Out(G)$.  
 For instance, it is  easy   to check that tree finiteness   holds for splittings of a finitely generated group
over the trivial group.   However, Example \ref{ex_Linnell} will show  that tree finiteness 
does not hold for splittings over $\bbZ/2\bbZ$ (although vertex finiteness holds by Theorem \ref{main}).

   Tree finiteness was established by Sela and Delzant  \cite[Corollary 4.9]{Sela_acylindrical}, \cite[Theorem 3.2]{Delzant_accessibilite}
    for virtually cyclic splittings of one-ended hyperbolic groups, using acylindrical super accessibility. We generalize their result as follows:

\begin{thm}\label{ttf}
Let $G$ be one-ended, and hyperbolic relative to virtually polycyclic groups. Up to the action of $\Out(G)$, there exist only finitely many minimal splittings of $G$ over virtually cyclic groups.
\end{thm}

\begin{example}\label{to}
In this example, tree finiteness does not hold for splittings of a one-ended toral relatively hyperbolic group over abelian groups, even if these groups are assumed to be  closed under taking roots.
Let $G$ be the free product of $A=\grp{a_1}\oplus\grp{a_2}\oplus\grp{a_3}\simeq\bbZ^3$
with three free groups $G_i=\grp{x_i,y_i}$, amalgamated along $[x_i,y_i]=a_i$.
For any $b\in \bbZ^3$, there is a one-edge splitting of $G$   over the abelian group $\grp{a_1,b}$,  with vertex groups $\langle G_1,b\rangle$ and $\langle G_2,G_3,A\rangle$. 
Since  $A$ is $\Aut(G)$-invariant (up to conjugacy), 
 and only finitely many automorphisms of $A$ extend to $G$, there is no tree finiteness. Note, however, that the isomorphism type of  $\langle G_1,b\rangle$ only depends on whether $b$ is a power of $a_1$ or not.
\end{example}
 
 Our motivation for Theorem \ref{main} was the study of automorphisms. In  \cite{GL_extension} we use Theorem \ref{main} to extend Shor's theorem \cite{Shor_Scott,LL_crelle} to toral relatively hyperbolic groups: up to isomorphism, there are only finitely many fixed subgroups of automorphisms. Theorem \ref{main} is also an important ingredient in our proof   that the set of McCool groups
of $G$ satisfies a bounded chain condition when $G$ is  
 toral relatively hyperbolic \cite{GL_McCool}
(a
 McCool group of $G$ is
  the subgroup of $\Out(G)$ fixing a given finite set of conjugacy classes of $G$).

Assertion 1 of Theorem \ref{main} is proved in Section \ref{feg}. The other assertions are proved simultaneously in later sections. 
 We  successively    consider    one-edge splittings of one-ended groups, then one-edge splittings of arbitrary groups, and finally   splittings with several edges.  Tree finiteness (Theorem \ref{ttf}) is proved at the end of  Section \ref{unbout}.
 
\paragraph{Acknowledgements}  
The first author acknowledges support from ANR-11-BS01-013. The second author acknowledges support from  ANR-10-BLAN-116-03.

\section{Preliminaries}

\subsection{Trees and splittings} \label{ts}
In this paper, $G$ will always denote a finitely generated group. 

A tree will be a simplicial tree $T$ with an action of $G$ without inversions. Two trees are considered to be the same if there is a $G$-equivariant isomorphism between them.

We usually assume that the action is \emph{minimal} (there is no proper invariant subtree) and that there is \emph{no redundant vertex} (if $T\setminus \{x\}$ has 2 components, some $g\in G$ interchanges them). 
 The tree $T$ is \emph{trivial} if there is a global fixed point (minimality then implies that $T$ is a point).  
An element of $G$, or a subgroup, is \emph{elliptic} if it fixes a point in $T$. 

An action of $G$ on  a tree $T$ gives rise to a splitting of $G$, i.e.\ a decomposition of $G$ as the fundamental group of  the quotient graph of groups $\Gamma=T/G$. Conversely, $T$ is the Bass-Serre tree of $\Gamma$. 
All definitions given here apply to both splittings and trees.

We usually restrict edge groups by requiring that they belong to  a family $\cala$  as in  Theorem \ref{main}.
We then say that the splitting is \emph{over groups in $\cala$}, or \emph{over $\cala$}.  
The group $G$ \emph{splits over  $A$} if $A$ is an edge group of   a non-trivial   splitting. 

There is a one-to-one correspondence between vertices (resp.\ edges) of $\Gamma$ and $G$-orbits of vertices (resp.\ edges) of $T$. We say that $\Gamma$ is a \emph{one-edge splitting} if it has exactly one edge.
We denote  
by $G_v$ the group carried by  a vertex $v$  of  $\Gamma$. We also view $v$ as a vertex of $T$ with stabilizer $G_v$. Similarly, we denote by $e$ an edge of $\Gamma$ or $T$, and by $G_e$ the corresponding group. The   groups carried by edges of $\Gamma$ incident to a given vertex $v$ will  be called the \emph{incident edge groups}  at $v$  (we usually view them as subgroups of $G_v$).

A tree $T'$ is a \emph{collapse} of $T$ if it is obtained from $T$ by collapsing 
each edge in a certain $G$-invariant
collection  
to a point; conversely, we say that $T$ \emph{refines} $T'$. 
In terms of graphs of groups, one passes from $\Gamma=T/G$ to $\Gamma'=T'/G$ by collapsing edges;
  for each vertex $v'$  of $\Gamma'$, 
the vertex group $G_{v'}$ is the fundamental group of the graph of groups $\Gamma_{v'}$ occuring as the preimage of $v'$ in $\Gamma$.

Conversely, suppose   $v'$ is a vertex  of a splitting $\Gamma'$, and $\Gamma_{v'}$ is  a splitting of $G_{v'}$ in which   incident edge groups are elliptic. One may then refine $\Gamma'$ at $v'$ using $\Gamma_{v'}$, so as to obtain a splitting $\Gamma$ whose edges are those of $\Gamma'$ together with those of $\Gamma_{v'}$. Note that $\Gamma$ is not uniquely defined because there is flexibility in the way edges of $\Gamma'$ are attached to vertices of $\Gamma_{v'}$;  this is   discussed in   
 Subsection  \ref{reft}.

All maps between trees will be $G$-equivariant. 
 Given two trees $T$ and $T'$, we say that $T$ \emph{dominates} $T'$ if there is a  
 map $f:T\to T'$, or equivalently if every subgroup which is elliptic in $T$ is also elliptic in $T'$.  In particular, $T$ dominates any collapse $T'$.  

Two trees  belong to the same \emph{deformation space} if they dominate each other. In other words, a deformation space $\cald$ (over $\cala$) is   the set of all trees (with edge stabilizers in $\cala$) having a given family of subgroups as their elliptic subgroups.  
All trees in a given deformation space over $\cala$ have the same set of vertex stabilizers, provided that one restricts to stabilizers not in $\cala$ \cite{GL2}.  We sometimes view a deformation space as a set of splittings (rather than trees).

Groups as in Assertions 2, 3, 4 of Theorem \ref{main} are accessible, so there exists a \emph{Stallings-Dunwoody deformation space}: it consists of trees with finite edge stabilizers  whose vertex stabilizers have at most one end. In the context of Assertion 1, we  shall consider the   
deformation space $\cald_k$ over $\fini_k$  consisting of trees whose vertex  stabilizers do not split over a group in $\fini_k$  (recall that a group is in $\fini_k$ if it has order $\le k$).  
These deformation spaces may (and should) be viewed as JSJ deformation spaces  over
the class of finite groups or 
over $\fini_k$  
respectively (see \cite{GL3a}).

A tree is \emph{reduced} if $G_e\ne G_v,G_w$ whenever an edge $e$ has its endpoints $v,w$ in different $G$-orbits 
 (being reduced in the sense of \cite{BF_bounding} is   a weaker property). Equivalently, no tree obtained from $T$ by collapsing the orbit of an edge belongs to the same deformation space as $T$. If $T$ is not reduced, one may collapse edges so as to obtain a reduced tree in the same deformation space.

\subsection {Virtually polycyclic groups}

We collect a few simple algebraic facts.  
 We write $ | X | $ for the cardinality of a finite set.

\begin{lem} \label{polycy}
Let $H$ be virtually polycyclic.
\begin{enumerate}
\item $H$ only contains finitely many conjugacy classes of finite subgroups.
\item  Given   a subgroup $A\inc H$, there exists a finite index subgroup $A_0\inc A$ such that, up to conjugation by an element of the normalizer $N(A_0)$, there exist    only   finitely many subgroups $B\inc H$ containing $A$ with finite index. 

\item Given   a subgroup $A\inc H$, there exists a number $N$ such that, if $B\inc H$ contains $A$ with finite index $n$, then $n\le N$. 
\end{enumerate}
\end{lem}

\begin{proof} The first assertion  is contained in Theorem 8.5 of \cite{Segal_livre}. It  is equivalent to  2 when $A$ is trivial. 

To prove 2 in general,   define $C(A)$ as the commensurator of $A$, equal to the set of $g\in G$ such that $gAg\m\cap A$ has finite index in $A$ and $gAg\m$. Note that any $B$ containing $A$ with finite index is contained in $C(A)$. Let $A_0=\cap_{g\in C(A)}gAg\m$. By \cite{Rhemtulla67}, $A_0$  is the intersection of a finite family of conjugates of $A$, so $A_0$ has finite index in $A$. It is normal in $C(A)$, and Assertion 2 follows by applying   1 to $C(A)/A _0$. In particular, there is a bound for the index of $A_0$ in $B$, so 3 is proved. 
\end{proof}

\begin{lem} \label{vc}
\begin{enumerate}
    
\item Given $n\in\N$, there are finitely many isomorphism types of virtually cyclic  groups $A$ such that all finite subgroups of $A$ have order $\le n$.
\item Given two virtually cyclic groups $A$ and $B$, and $n\in\N$, there are   only finitely many   monomorphisms  $i:A\to B$ such that the index of $i(A)$ in $B$ is $\le n$, up to precomposition by an inner automorphism of $A$.

\end{enumerate}
\end{lem}

\begin{proof} An infinite virtually cyclic group $A$ maps with finite kernel $N$ onto a group  
  which is 
  either infinite cyclic or equal to the infinite dihedral group $D_\infty$ 
(see \cite[Theorem 5.12]{ScottWall}).   
In the first case, $A$  is a semidirect product  $N\rtimes \bbZ 
$ and 
there are only finitely many possibilities for $A$ up to isomorphism since $ | N | $ is bounded.
In the second case, $A\simeq N_1*_N N_2$ with $ | N_1 | = | N_2 | =2 | N | $, and again there are only finitely many possibilities.
For the second assertion, note that there are finitely many possibilities for the image of $i$. Two injections with the same image differ by an automorphism of $A$, and $\Out(A)$ is finite. 
\end{proof}

\begin{lem} \label{sand} 
 Fix  a finitely generated abelian group $P$, and  a subgroup $A\inc P$.
Say that two subgroups $B,B'$ with $A\inc B\inc P$ and $A\inc B'\inc P$ are equivalent if there is an isomorphism $B\to B'$ equal to the identity on $A$. 

Then the number of equivalence classes is finite. 
\end{lem}

\begin{proof} 
Define the root-closure  $e( A,B)$ as the set of  elements of $B$ having a power in $A$.  
It contains the torsion subgroup of $B$, and it is 
the smallest subgroup  of  $B$ containing $A$ and such that $B=e(A,B)\oplus B_0$ with $B_0\subset B$   torsion-free.
Equivalently, $e(A,B)$ is
the largest subgroup of $B$ containing $A$ with finite index. 
Note that $A\inc e( A,B)\inc e(A,P)$, with all indices finite. 
As $B$ varies, there are only finitely many possibilities for $e(A,B)$, and for the isomorphism type of   $B_0$. 
  When $e(A,B)=e(A,B')$, and  $B_0\simeq B'_0$,
any isomorphism $B_0\ra B'_0$
extends to an isomorphism $B\to B'$ equal to the identity on $e(A,B)$, hence on $A$.
\end{proof}

\begin{cor}\label{cor_is}
Fix two groups  $ G_0$ and $P$ with a common subgroup $A$, where $P$ is finitely generated abelian.
As $B$ varies among subgroups such that $A<B<P$, 
  the groups $G_0*_A B$
lie in finitely many isomorphism classes.
\end{cor}

  Indeed,  $G_0*_A B\simeq G_0*_A  B'$ if $B,B'$ are equivalent.

\subsection{Relatively hyperbolic groups}

Suppose that $G$ is  as in Assertion 2 or 3   of Theorem \ref{main}, \ie 
$G$ is hyperbolic relative to a finite family $\{P_1,\dots, P_k\}$ of finitely generated subgroups, 
  which are virtually polycyclic or abelian. 
Subgroups of $P_i$, and their conjugates, are called \emph{parabolic}.  A subgroup of $G$  is  virtually polycyclic if and only if it is  
 parabolic or virtually cyclic. 
 
Any infinite virtually polycyclic subgroup is contained in a unique maximal one, which is virtually cyclic (loxodromic) or conjugate to some $P_i$.   Such a maximal subgroup $H$  is \emph{almost malnormal}: if $gHg\m\cap H$ is infinite, then $g\in H$.

\begin{lem} \label{relpolycy}
Let $G$ be hyperbolic relative to virtually polycyclic subgroups. 
\begin{enumerate}
  \item  $G$ only contains finitely many conjugacy classes of finite subgroups.
\item Up to isomorphism, $G$ only has   finitely many virtually cyclic subgroups. 
\item  Given  a virtually polycyclic subgroup $A\inc G$, there are only finitely many groups $B\inc G$ containing $A$ with finite index, up to conjugacy in $G$; when $B$ varies, the index of $A$ in   $B$ remains bounded. 
\end{enumerate}
\end{lem}

\begin{proof} In a  relatively hyperbolic group,  all finite subgroups outside of a finite number of conjugacy classes 
are parabolic 
(see for instance Lemma 3.1 of \cite{GL6}),   so Assertion 1 follows from Lemma \ref{polycy}. 
Assertion 2  follows from Lemma \ref{vc}. Assertion 3 is clear if $A$ is finite or loxodromic, and follows from Lemma \ref{polycy} otherwise.
\end{proof}

Note that the lemma also holds if $G$ is a CSA group as in Assertion 4. In this case, all virtually polycyclic subgroups are abelian.

\subsection{About the proofs}\label{rks}

The next four sections are  devoted to the proof of Theorems \ref{main} and \ref{ttf}. All splittings will be over groups in the relevant family $\cala$.  

Note that, under all assumptions, 
groups in  $\cala$  are virtually abelian and  fall into finitely many isomorphism classes (this
 follows from Lemma \ref {relpolycy}  in the  relatively hyperbolic case).  
It therefore  suffices to prove  vertex finiteness for  reduced splittings 
since vertex groups not in $\cala$ remain when one collapses edges to  obtain a reduced splitting in the same deformation space.
Another consequence is  that vertex finiteness in fact holds for non-minimal splittings.

We also  note that vertex groups of splittings of $G$ over $\cala$ satisfy the assumptions of the theorem (they are finitely generated, relatively hyperbolic, or CSA). 
In the  relatively  hyperbolic case,  this follows from Theorem 1.3 of \cite{Bo_peripheral}, as explained in the proof of Theorem 3.35 of \cite{DaGr_isomorphism} (for Assertion 3, note that nonabelian virtually cyclic groups may be removed from the list of maximal parabolic subgroups);  
in the CSA case, vertex groups  are finitely generated (because edge groups are) and CSA. This makes inductive arguments possible.

A basic method for showing vertex finiteness is to  represent 
any vertex group $G_v$ as
the fundamental group of another graph of groups  whose number of edges is bounded, 
and where the set of possible isomorphism types of edge and vertex groups is finite. 
One then has to control  inclusions of edge groups into vertex groups. 

When edge groups are finite, it suffices to know that vertex groups only contain finitely many conjugacy classes of finite subgroups, 
since postcomposing an inclusion $G_e\to G_v$ with an inner automorphism of $G_v$ does not change the fundamental group of the graph of groups.

 When edge groups are infinite  and $G$ is one-ended, we use   
a canonical JSJ decomposition $\Gcan$, and   its \emph{universal compatibility} with the splittings considered: 
given any $\Gamma$, there is a splitting $\Lambda$ such that both $\Gamma$ and $\Gcan$ may be obtained from $\Lambda$ by collapsing edges.

\section{Splittings over finite  groups}
\label{feg}

We prove the first assertion of Theorem \ref{main}.
All splittings considered here will be  minimal and over groups belonging to $\fini_k$, the family of all subgroups of order $\le k$. Linnell's accessibility \cite{Linnell} provides a bound (depending on $G$ and $k$) for the number of edges of such splittings, as long as the splittings have no redundant vertex. 
 
We  shall first show:

\begin{lem} \label{sk2} Let $G$ be a finitely generated group,  $k\ge1$, and $\cald$ a deformation space over $\fini_k$. 
Then $\cald$ only contains  finitely many reduced trees  $T$ up to the action of $\Out(G)$. 
\end{lem}

More precisely: the subgroup $\Out(\cald)$ of $\Out(G)$ consisting of automorphisms leaving $\cald$ invariant acts on the set of reduced trees in $\cald$ 
with finitely many orbits. 

The   example below shows that $\Out(\cald)$ does not always act on the whole of $\cald$ with finitely many orbits, 
because  of non-reduced trees.  
It also shows that the number of deformation spaces  of $G$ over $\fini_2$  may be  
 infinite modulo $\Out(G)$.  In particular,  tree finiteness does not hold for splittings over $\fini_2$ (though it holds for splittings  over the trivial group).

\begin{example}\label{ex_Linnell}  Let $A$ be a one-ended group whose set of elements of order $2$ is not  a finite union of $\Aut(A)$-orbits (one can check that the lamplighter group  $(\Z/2\Z)\wr \bbZ$ 
is such a group, cf.\  Proposition 2.1 of \cite{GoWo_wreath}).
 Let $G=A*B$, with $B$ one-ended. 
Let  $\cald$ be the deformation space over $\fini_2$ containing the Bass-Serre tree of the defining 
 free product $G=A*B$. Here, $\Out(\cald)=\Out(G)$.
For any subgroup $F<A$ of order $2$, the Bass-Serre tree $T_F$ of the (non-reduced) two-edge graph of groups decomposition $G=A*_F F * B$ lies in $\cald$,
and  the trees $T_F$ are not contained in a finite union of $\Out(G)$-orbits as $F$ varies.
Moreover,  
the  one-edge splittings  $G=A*_F( {F*B})$
define  infinitely many   $\Out(G)$-orbits of deformation spaces as $F$ varies.
\end{example}

\begin{proof}[Proof of Lemma \ref{sk2}] 
Let $\Gamma=T/G$. By   accessibility,   the number of edges of $\Gamma$ is bounded.   To describe $\Gamma$, we need to know edge groups, vertex groups, and inclusions of edge groups into vertex groups. There are only finitely many possibilities for edge groups (up to isomorphism).
Vertex groups of $\Gamma$ with order $>k$ do not depend on $\Gamma$ for $T\in\cald$,
 so there are only finitely many possibilities for vertex groups of $\Gamma$ up to isomorphism.
To prove finiteness, it therefore suffices to show that there are only finitely many possibilities for the image of an edge group in a vertex group    of order $>k$, up to conjugacy. 

Fix a reduced $\Gamma_0=T_0/G$, with $T_0\in\cald$.  Since no edge group may be properly contained in a   conjugate of itself,   it follows from Proposition 4.9 of \cite{GL2} that  any vertex group $H$  of   $\Gamma_0$, with $H \notin\fini_k$,  contains  finitely many  subgroups  $E_i(H)\in \fini_k$ with the following property: given any \emph{reduced} $\Gamma=T/G$ with $T\in\cald$, each incident edge group  of the vertex group   $G_v=gHg\m$ of $\Gamma$ conjugate to $H$ is contained in a $G_v$-conjugate of some  $gE_i(H)g\m$. 
  The required finiteness follows since any vertex group   of $\Gamma$ of order $>k$ is conjugate to some $H$.
\end{proof}

\begin{rem} \label{retr}
 In Section 7 of \cite{GL2}, we have defined an  $\Out(\cald)$-invariant retract $\calg\inc\cald$. It consists of   trees $T \in\cald$ all of whose edges are surviving edges:   given any edge $e$, one can collapse $T$ to a reduced tree $T'\in\cald$ without collapsing $e$.  The same argument  as above shows that $\calg$  only contains  finitely many  trees, up to the action of  $\Out(\cald)$.
  This says that $\calg/\Out(\cald)$ is a finite complex with missing faces, or
equivalently that its spine is finite (see \cite{GL2}).

If $G$ is accessible, we can consider the Stallings-Dunwoody deformation space $\cald$ and its retract $\calg$. 
Edge stabilizers of trees   in $\calg$ all belong to some fixed $\fini_k$, so  
$\calg$ coincides with the retract of a deformation space over $\fini_k$, and $\calg/\Out(G)$ is  finite as above. 
\end{rem}

We now prove the first assertion of Theorem \ref{main}. By Linnell's accessibility, there is a (unique) 
  deformation space $\cald_k$  over $\fini_k$ such that vertex  stabilizers of trees in $\cald_k$ do not split over a group in $\fini_k$ (this is the JSJ deformation space over $\fini_k$, see \cite{GL3a} subsection 6.3).
   Recall that all  trees in $\cald_k$ have the same vertex  stabilizers of order $>k$.

Let $H=G_v$ be a vertex group of  a      splitting $\Gamma$  of $G$ over $\fini_k$.   As explained in Subsection \ref{rks}, we may assume that $\Gamma$ is reduced. By Lemma 4.8 of \cite{GL3a}, one may refine $\Gamma$ to a (JSJ) splitting $\Lambda$ in $\cald_k$. 
The refinement replaces the vertex $v$ by a subgraph of groups $\Lambda_v\inc \Lambda$ whose fundamental group is $H$. We may assume that $\Lambda_v$ is reduced (but not that $\Lambda$ is). We show that there are only finitely many possibilities for $\Lambda_v$. 

  The splitting $\Lambda$ is not necessarily reduced, so let $p:\Lambda\to\Lambda_0$ be a collapse map to a reduced  splitting in $\cald_k$. 
Since  $\Lambda_v$ is reduced,  the map $p$ does not collapse  any edge coming from $\Lambda_v$. In particular, the number of  
edges of $\Lambda_v$
is bounded by the number of edges of $\Lambda_0$, which is bounded by Lemma \ref{sk2}. 

 Vertex groups of $\Lambda$ of order $>k$ are vertex groups of $\Lambda_0$, so
there are finitely many possibilities for vertex and edge groups of $\Lambda_v$ up to isomorphism.
There remains to control inclusions $G_e\to G_u$ from edge groups of $\Lambda_v$ to vertex groups.  
We may assume that $G_u$ has order $>k$. This implies that the group 
carried by $p(u)$ in $\Lambda_0$ is $G_u$  (but the group carried by the other endpoint of $e$ may grow). 
Finiteness follows from the finiteness of possible images of incident edge groups in vertex groups  of graphs in $\cald_k$ as in the previous proof.

\section{One-edge splittings of one-ended groups} \label{unbout}

In this section we prove Assertions 2, 3, 4 of Theorem \ref{main} for one-edge splittings of one-ended groups. We will also prove  Theorem \ref{ttf} (see Subsection \ref{tf}). 

We   assume that $G$ is one-ended and we consider a one-edge splitting $\Gamma$.  As explained in Subsection \ref{rks}, we may assume that $\Gamma$ is  reduced (\ie minimal).
All splittings will be over groups in $\cala$ (necessarily infinite by one-endedness).   

We first explain how to obtain $\Gamma$  from a JSJ decomposition $\Gcan$ by refining and collapsing. We then discuss refining in general (Subsection \ref {reft}).   

\subsection{The canonical JSJ splitting}\label{cJSJ}
 
Let $T_c^*$
be the canonical JSJ tree over $\cala$ constructed in Theorems 11.1 and 13.1 of \cite{GL3b} (applied with $\calh=\es$), and $\Gcan$ the associated graph of groups.  In all cases  considered here it is the JSJ decomposition of $G$ over $\cala$ relative to all virtually polycyclic subgroups which are not virtually cyclic.
We summarize  the  relevant properties of $\Gcan$. 

 $\Gcan$ is not necessarily reduced (and may have redundant vertices). Its vertex groups $G_v$ are either maximal virtually polycyclic subgroups, or rigid, or QH with finite fiber.  If $G_v$ is rigid or QH, incident edge groups 
 $G_e$ are  maximal virtually abelian subgroups of $G_v$.

 If $G_v$ is rigid, it has no non-trivial splitting over groups in $\cala$ in which incident edge groups are elliptic. If $G_v$ is QH, there is an exact sequence $1\to F\to G_v\to\pi_1(\Sigma_v)\to 1$ where the fiber $F$ is finite and $\Sigma_v$ is a compact 2-dimensional orbifold. Incident edge groups are preimages of boundary subgroups of $\pi_1(\Sigma_v)$ (i.e.\ fundamental groups of boundary components of $\Sigma_v$),
 and conversely such preimages are incident edge groups (up to conjugacy). 

\begin{figure}[htbp]
  \centering
  \includegraphics{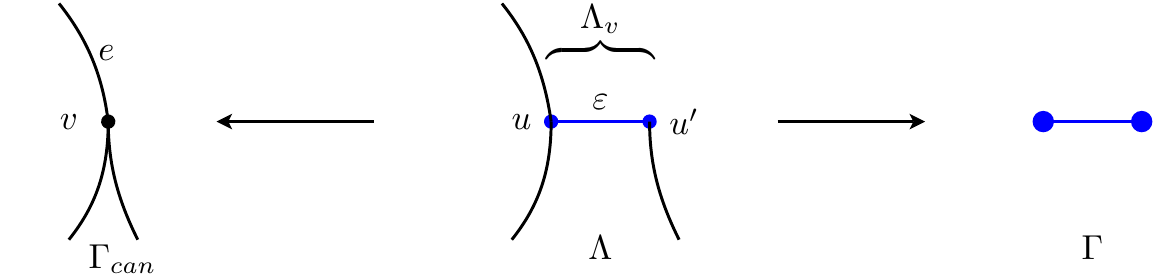}
  \caption{The common refinement $\Lambda$ of $\Gcan$ and $\Gamma$.}
  \label{fig_blowup}
\end{figure}

Moreover, $\Gcan$ is universally compatible. This means that, 
  given a  non-trivial one-edge splitting $\Gamma$ as above, there is a splitting   $\Lambda$ 
which collapses onto both $\Gamma$ and $\Gcan$. 
 It is minimal, but not necessarily reduced. After collapsing   edges in $\Lambda$, we may assume that no edge of $\Lambda$ is collapsed in both
$\Gamma$ and $\Gcan$.
Let $\eps$ be the edge of $\Lambda$ that is not collapsed in $\Gamma$  (see Figure \ref{fig_blowup}).
We can assume that $\eps$ is  collapsed in $\Gcan$, since otherwise  
$\Gamma$ is a collapse of $ \Gcan$,
and this  only produces finitely many  splittings.

Denote by $v$  the vertex of $\Gcan$ to which $\eps$ is collapsed,  and by $G_v$ the corresponding vertex group.  Let
  $\Lambda_v\subset \Lambda$ be the one-edge splitting  of $G_v$ associated to    $\eps$,
so that $\Lambda$ is obtained from $\Gcan$ by replacing the vertex $v\in \Gcan$ 
by the one-edge decomposition $\Lambda_v$ of $G_v$.

Note that $\Lambda_v$ can be a trivial decomposition 
(i.e.\ an amalgam of the form $G_v=G_v*_{G_\eps} G_\eps$).
This occurs precisely when $\Gcan$ and $\Lambda$ belong to the same deformation space.
In this case, the splitting $\Lambda$ is not  reduced.

\subsection{Refining a splitting}\label{reft}

Knowing $\Gcan$ and $\Lambda_v$  
is not enough to determine $\Lambda$ and $\Gamma$: 
one must also know how edges $e$ of $\Gcan$ incident to $v$ are attached to vertices of $\Lambda_v$  
(note that refining is possible only if all groups $G_e$ are elliptic in $\Lambda_v$). 
When $\Lambda_v$ has two vertices, one must first decide to which vertex $u$  of $\Lambda_v$ each edge $e$ is attached. This is a combinatorial choice, with only finitely many possibilities, so we will always assume that this choice has been made. One must then know, for each $e$, the injection of $G_e$ into $G_u$, and this is a possible cause of infiniteness. We demonstrate this on an example.

\subsubsection{Changing attachments}
\label{sec_heis} 

We construct a splitting $\Theta_0$ of a group $G$ such that  \emph{there are infinitely many ways to refine the Bass-Serre tree of $\Theta_0$ using a fixed one-edge splitting $\Lambda_v$ of a vertex group $G_v$} (see Figure \ref{fig_attachments}). 
This will also demonstrate that there is no vertex finiteness for abelian splittings of groups which are hyperbolic relative to nilpotent groups.

\begin{figure}[htbp]
  \centering
  \includegraphics[width=\textwidth]{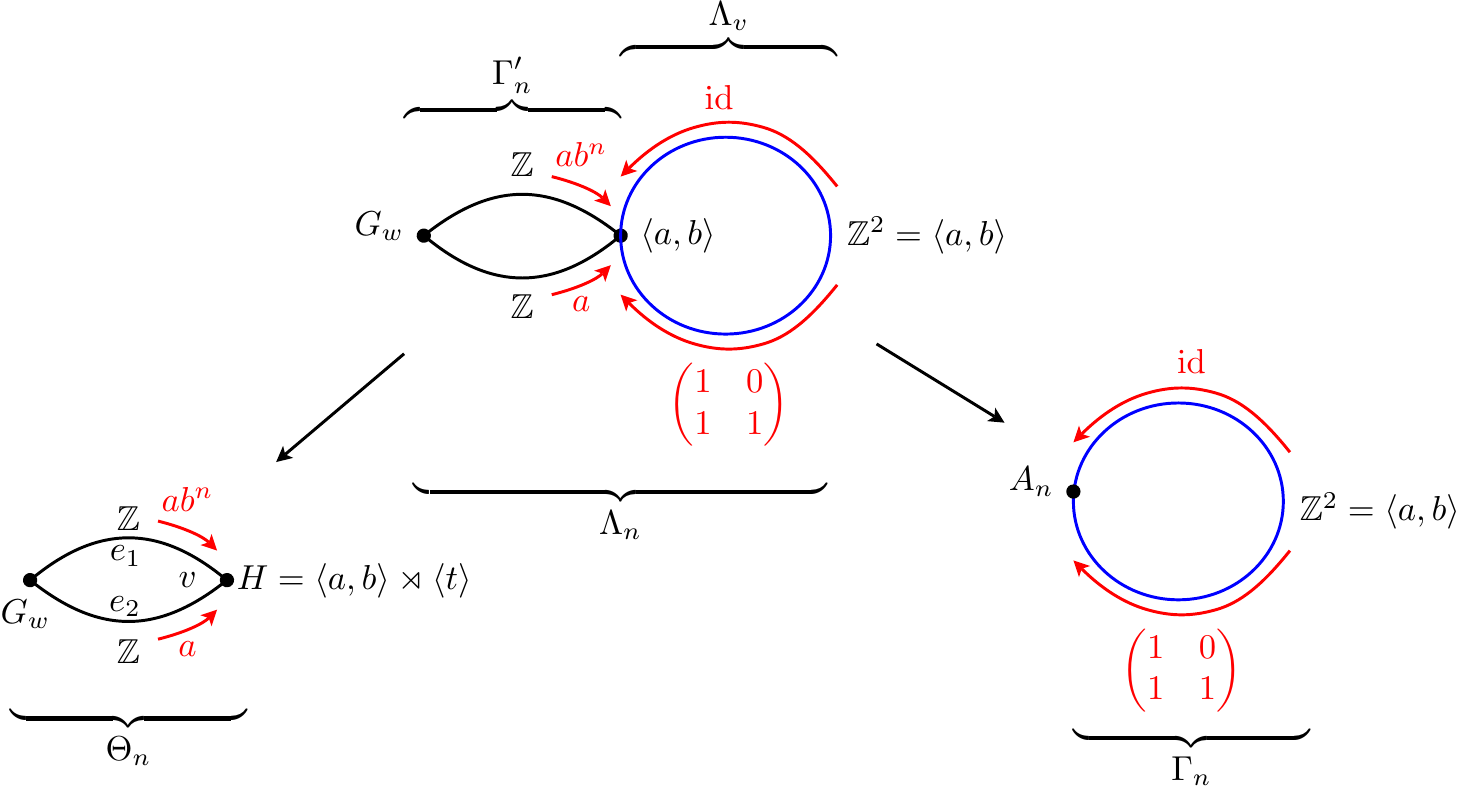}
  \caption{Infinitely many refinements of a graph of groups }
  \label{fig_attachments}
\end{figure}

 Let $H$ be the Heisenberg group, which we view as a semidirect product $\Z^2\rtimes\Z=\langle a,b,t\mid ab=ba, tat\m=ab, tbt\m=b\rangle$. The splitting $\Theta_0$ has two vertices $v,w$, with $G_v=H$  and $G_w$ a torsion-free hyperbolic group with no non-trivial cyclic splitting. They are joined by two edges $e_1,e_2$ carrying infinite  cyclic groups. The inclusions of edge groups into vertex groups 
map both $G_{e_1}$ and $G_{e_2}$ onto $\langle a\rangle$ in $H=G_v$, and they  map $G_{e_1},G_{e_2}$ onto non-conjugate maximal cyclic subgroups of $G_w$. The splitting $\Lambda_v$ of $G_v=H$ is the HNN extension associated to the semidirect product.
The group $G=\pi_1(\Theta_0)$ is hyperbolic relative to   the nilpotent group $H$ by \cite{Dah_combination}, and it may be checked that $\Theta_0$ is its 
 JSJ decomposition over abelian (or nilpotent)
groups relative to $H$.   

Let $\Lambda_0$ be obtained by refining $\Theta_0$ using $\Lambda_v$ in the obvious way.  Collapsing $e_1,e_2$ in $\Lambda$ yields an HNN extension $\Gamma_0$ with edge group $\Z^2=\langle a,b\rangle$. The base group $A_0$ is the fundamental group of the graph of groups  $\Gamma'_0$ obtained from   $\Theta_0$ by making  
 the group carried by $v$ equal to $\Z^2=\langle a,b\rangle$ rather than $H$.
 
Now let $n\in\N$. Consider $\Theta_0$ and define a new graph of groups $\Theta_n$ by
postcomposing   the inclusion   $G_{e_1}\to G_v$
with conjugation by $t^n$, an inner automorphism of $G_v$;  the image of $G_{e_1}$ is now generated by $t^nat^{-n}=ab^n $. 
  Since we changed the edge  monomorphism  by an inner automorphism of the vertex group,
$\Theta_n$ and $\Theta_0$ are equivalent (they are associated to the same Bass-Serre tree). 
Then construct $\Lambda_n$, $\Gamma_n$   and $\Gamma'_n$ as above.  

It is still true that $\Lambda_n$ refines $\Theta_n$, and the base group $A_n$ of the HNN extension $\Gamma_n$ is   the fundamental group of a graph of groups $\Gamma'_n$ with vertices carrying $\Z^2=\langle a,b\rangle$ and $G_w$.
But the inclusion of $G_{e_1}$ into $\Z^2$ now has image generated by $ ab^n$.   In particular, the subgroup of $\Z^2$ generated by incident edge groups  in $\Gamma'_n$ is $\langle a,ab^n\rangle$, it has index $n$. This shows that the splittings $\Gamma'_n$ (hence also the $\Lambda_n$'s) are distinct. Moreover, $\Gamma'_n$ is the canonical cyclic JSJ decomposition of $A_n$ relative to non-cyclic abelian groups, so  the   $A_n$'s are pairwise non-isomorphic.

In terms of trees, the minimal $H$-invariant subtree in the Bass-Serre tree  $T_0$ of $\Lambda_0$ is a line $L$. There are lifts of $e_1$ and $e_2$   attached to vertices of $L$. In the Bass-Serre tree  $T_n$ of $\Lambda_n$, the attachment point of a given lift  of $e_1$ gets  ``shifted'' by a translation of length $n$ along $L$: if 
$\tilde e_1,\tilde e_2$ are lifts of $e_1, e_2$ to $T_n$  having the same stabilizer, then their distance in $T_n$ is $n$.

\subsubsection{Practical description of a one-edge refinement}
We now explain how to describe  all one-edge refinements $\Lambda$ of a given graph of 
groups  $\Theta$ at a vertex $v$ (i.e.\ $\Lambda$ is obtained by refining  $\Theta$ at $v$ using a one-edge splitting). In the next subsection, we will take $\Theta$ to be the canonical JSJ decomposition $\Gcan$. We view $G_v$ as a subgroup of $G$, and, for each edge $e$ incident to $v$ in $\Theta$, we view $G_e$ as a subgroup of $G_v$.

By Bass-Serre theory, a graph of groups $\Lambda$ gives an action of a group $G_\Lambda$ on a tree $T_\Lambda$. We consider $\Lambda$ and $\Lambda'$  as  \emph{equivalent} if there is an isomorphism $\tau:G_\Lambda\to G_{\Lambda'}$ and a  $\tau$-equivariant isomorphism $T_\Lambda\to T_{\Lambda'}$.

\begin{lem}\label{lem_notations}
Up to equivalence, any  one-edge refinement $\Lambda$ of $\Theta$ at a vertex $v $ may be  obtained from the following data:
  \begin{enumerate}
  \item (marked splitting): an isomorphism $\phi:G_v\to\pi_1(\Lambda_v)$,
where $\Lambda_v$ is a
    one-edge splitting  
    (which may be a  trivial splitting $G_v*_{G_\varepsilon}G_\varepsilon$);
  \item (combinatorial attachment): when $\Lambda_v$ is an amalgam,  the choice of a vertex
    $u_e$ of $\Lambda_v$ for each oriented edge $e$ of $ \Theta$ incident to $v$;  
      \item (algebraic attachment): for each oriented edge $e$ of $ \Theta$ incident to $v$, 
    a   monomorphism $i_e:G_e\ra \varphi\m (G_{u_e})$ which is the restriction of   
     some inner automorphism $\ad_{g_e}\in \Inn(G_v)$.
  \end{enumerate}
\end{lem}

Different data may yield equivalent splittings. For instance, postcomposing $i_e$ with an inner automorphism of $\varphi\m (G_{u_e})$ does not change $\Lambda$.

\begin{proof}
Starting from the data, one
 constructs a graph
  of groups $\Lambda$ as follows.  The underlying graph is obtained from that of $\Lambda$ by blowing up $v$ into the one-edge graph underlying $\Lambda_v$, and attaching incident edges as prescribed by the combinatorial attachment data.

  The vertex groups are those of $\Theta\setminus \{v\}$, and preimages under $\varphi$ of those of 
  $\Lambda_v$; 
   the   edge groups are those of $\Theta$, and the preimage of the edge group of
  $\Lambda_v$; the   monomorphisms from edge groups to vertex groups     are the natural ones (those of $\Theta$ and $\Lambda_v$) and the $i_e$'s.
  Collapsing the edge of $\Lambda_v$ yields $\Theta$ (up to equivalence) because of 
  the  requirement that $i_e$ be the restriction of an inner automorphism.
  
Conversely,
  if $\Lambda$ is a one-edge refinement of $\Theta$, with Bass-Serre tree $T_\Lambda$, one defines $\Lambda_v$ as the one-edge splitting associated to the $G_v$-invariant subtree $T_v\inc T_\Lambda$ which is collapsed to a point $\tilde v$ in the Bass-Serre tree $T_{\Theta}$ of $\Theta$. 
  We fix an identification $\varphi
$ by choosing an edge in   $T_v$. In particular, this selects a vertex in each $G_v$-orbit of vertices of $T_v$ (there is one or two orbits, so one or two selected points $u,u'$). 

If  $e$ is an edge of $\Theta$ incident to $v$,
we view $G_e$ as the stabilizer of an edge $\tilde e$ of $T_{\Theta}$ incident to $\tilde v$. 
In $T_\Lambda$, this edge is attached to a vertex $v_e$  of $T_v$. The orbit of  $v_e$  determines the combinatorial attachment, 
 and $i_e$ is induced by $\ad_{g_e}$ with $g_e$   any element of $G_v$ taking this vertex $v_e$ to the selected vertex $u$ or $ u'$.
 \end{proof}

\begin{rem}\label{lesi}
  If we replace a marking $\phi:G_v\ra \pi_1(\Lambda_v)$ by $\varphi'=\varphi\circ \psi$, with $\psi$ an inner automorphism  of $G_v$, the  refinements of $\Theta$ by $\Lambda_v$ obtained using $\varphi'$  are (up to equivalence) the same as those obtained using $\varphi$  (one simply replaces $i_e$ by  $\psi\m\circ i_e$).  This holds, more generally, if $\psi $ acts on each incident edge group $G_e$ as conjugation by some $g_e\in G_v$.  
\end{rem}

\begin{rem}\label{lesi2}
If   $\phi\m(G_{u_e})$ is  almost malnormal in $G_v$, and $G_e$ is infinite, 
then the different choices for $i_e$ differ by an inner automorphism of $\phi\m(G_{u_e})$
and therefore lead to equivalent splittings $\Lambda$.
 The same conclusions hold if $G_e$ is contained in a unique conjugate of
$\phi\m(G_{u_e})$ in $G_v$, and $\phi\m(G_{u_e})$ is its own normalizer in $G_v$.
\end{rem}

\subsection 
{Vertex finiteness over virtually cyclic  groups}
\label{vfvc}

This subsection is devoted to the proof of:

\begin{prop}\label{cor_cv1b1a}
Let $G$ be one-ended, and hyperbolic relative to virtually polycyclic groups.
Then vertex finiteness holds for one-edge virtually cyclic splittings of $G$. 
\end{prop}

 We will actually prove:
\begin{lem}\label{lem_TF1} 
Let $G$ be one-ended, and hyperbolic relative to virtually polycyclic groups. Let 
$\Theta$ be a virtually cyclic splitting of $G$, and $v$ a vertex.
Up to the action of $\Out(G)$, there exist only finitely many minimal virtually cyclic splittings $\Lambda$ obtained by refining $\Theta$ at $v$ using a one-edge splitting  $\Lambda_v$ with the following property: if $G_v$ is not 
virtually polycyclic or QH with finite fiber (as defined in Subsection \ref{cJSJ}), then $\Lambda_v$ is a trivial amalgam  $G_v*_{G_\varepsilon}G_\varepsilon$.
\end{lem}

   The lemma implies the proposition  because, as explained in Subsection \ref{cJSJ}, any one-edge virtually cyclic splitting of $G$ is a collapse of a one-edge refinement of $\Theta=\Gcan$; vertex groups   of $\Gcan$ which are not 
virtually polycyclic or QH with finite fiber are rigid, so can only be refined using a trivial amalgam.

 In Subsection \ref{tf} we will explain that the lemma 
 yields tree finiteness  
 for one-edge virtually cyclic splittings, and we will use it to prove  
 Theorem \ref{ttf}   (tree finiteness for arbitrary virtually cyclic  splittings). 
 
 \begin{proof}[Proof of Lemma  \ref{lem_TF1}]
  We use the notations of Lemma \ref{lem_notations},  and we denote by $G_\varepsilon$ the edge group of $\Lambda_v$.
We   assume (when $\Lambda_v$ has two vertices) that the combinatorial choice (deciding to which vertex of $\Lambda_v$ edges of $\Theta$ incident to $v$ will be attached) has been made. We must prove that varying   
the marked splitting and the $i_e$'s does not produce infinitely many $\Lambda$'s.
We distinguish  several cases, depending on the nature of $\Lambda_v$.

$\bullet$
First  suppose that $\Lambda_v$ is a trivial amalgam.  In this case we may assume that $\Lambda_v$ is 
$G_v=G_v*_{G_\eps} G_\eps$  for some virtually cyclic $G_\varepsilon\inc G_v$,  and $\varphi$ is the identity (the marked splitting is determined by $G_\varepsilon$, and changing $\varphi$ amounts to changing  $G_\varepsilon$ by an automorphism of $G_v$). 
Call $u,u'$ the vertices of $\Lambda_v$, with vertex groups $G_{u }=G_\eps$ and $G_{u'}=G_v$. 

 By minimality of $\Lambda$, at least one edge $e_1$ of $\Theta$ is attached to $u$. Thus $G_{e_1}$ is contained in a conjugate of $G_u=G_\varepsilon$. The groups $G_{e_1}$ and $G_\varepsilon$ are both infinite and virtually cyclic, so the index is finite. Since  by Remark \ref{lesi} the set of refinements does not change if we   replace $G_\varepsilon$ by a conjugate,
  Assertion 3 of Lemma \ref{relpolycy} (applied in $G_v$) lets us assume that  $G_\varepsilon$  
 is fixed (because there are only finitely many possibilities for $G_\varepsilon$ up to conjugacy). 
 
 We must now vary the maps $i_e$. 
For edges $e$ attached to $u'$, the choice of $i_e$ is irrelevant. For edges attached to $u$, Lemma \ref{relpolycy}  provides a bound for the index  $[G_u:i_e(G_e)]=[\ad_{g_e}\m( G_u): G_e]$.   By Assertion 2 of Lemma \ref{vc}, there are only finitely many possibilities for $i_e$  (up to inner automorphisms of $G_e$ or $G_u$).

$\bullet $  Now suppose that  $\Lambda_v$ is  non-trivial and  
 $G_v$ is virtually polycyclic.
The Bass-Serre tree of $\Lambda_v$ is a line $L$ on which $G_v$ acts 
by translations or dihedrally ($\Lambda_v$ is an HNN extension or an amalgam accordingly).
Note that, since $G_\eps$ is virtually cyclic, $G_v$ is virtually $\bbZ^2$.

The group   $G'_\eps=\phi\m(G_\varepsilon)\inc G_v$ is the set of elements acting as the identity on $L$. 
It is normal, with quotient   $\Z$ or the 
  infinite dihedral group $D_\infty$.   Given an incident edge group $G_{e}<G_v$ in $\Theta$,
the intersection $G'_\eps\cap G_{e}$ has index at most $2$ in $G_{e}$,  hence finite index in $G'_\eps$,  
so by Lemma \ref{polycy} there are only finitely many possibilities for  $G'_\eps$ as $\Lambda_v$ and $\varphi$ vary. 

Once $G'_\varepsilon$ is fixed, $\Lambda_v$ is
determined (up to  an equivalence, which is not necessarily relative to the incident edge groups)
because $\bbZ$ and $D_\infty$ only have one non-trivial
one-edge splitting.
This means that we can fix $\Lambda_v$ and a marking $\varphi_0$, and restrict to markings given by $\phi=\phi_0\circ \psi$ for some automorphism $\psi$ of $G_v$
preserving $G'_\eps$.

Recall (Remark \ref{lesi})  that the set of refinements associated to $\varphi$ (obtained by varying the $i_e$'s) is the same as for $\varphi_0$ if $ \psi$ acts on each incident edge group $G_e$ as conjugation by some $g_e\in G_v$.  We claim that the group consisting of the automorphisms  with this property has finite index in the group $A$ of all automorphisms $ \psi$ of $G_v$ preserving  $G'_\eps$.  

Given an  incident edge group $G_e$, one has $[G'_\eps:\psi(G_e\cap G'_\eps)]=[G'_\eps:(G_e\cap G'_\eps)]<\infty$, 
so $\psi(G_e\cap G'_\eps)$ takes only finitely many values as $\psi$ varies  in $A$.
By Lemma \ref{polycy}, there are only finitely many possibilities for $\psi(G_e)$ up to conjugacy in $G_v$.  
Replacing $A$ by finite index subgroups, we may arrange that $\psi(G_e)$ be conjugate to $G_e$, and then that $\psi_{ | G_e}$ be the restriction of an inner automorphism of $G_v$ since $\Out(G_e)$ is finite. Arguing in this way for each incident edge proves the claim.

The claim lets us   assume that $\phi$ is fixed.
The last thing to do is to   vary  the maps $i_e$. Finiteness is proved as in the previous case, since 
vertex groups of $\Lambda_v$ are virtually cyclic.

$\bullet  $ In the remaining case, 
$v$  is   a QH vertex of $\Theta$ and $\Lambda_v$ is  non-trivial. 
Then   (see \cite[Lemma 7.4]{GL3a})
$\Lambda_v$ is dual to a simple closed 1-suborbifold $\gamma$  
on the underlying 2-orbifold   $\Sigma_v $  (if $G$ is torsion-free, $\gamma$ is a curve on a surface). 
Up to a homeomorphism $f$ of $\Sigma _v$ equal to the identity on $\partial\Sigma_v $, 
there are only finitely many possible $\gamma'$s. 

First suppose that $G$ is torsion-free. Then $\Sigma_v$ is a surface and any $f$ as above induces an automorphism of $G_v$  acting on incident edge groups as a conjugation.  By Remark \ref{lesi}, we may therefore assume that the  marked splitting  
$\varphi:G_v\ra\pi_1(\Lambda_v)$ is fixed.  
The choice of the $i_e$'s is irrelevant by Remark \ref{lesi2}.

The argument is the same if there is torsion, noting 
that the group of automorphisms of $\pi_1(\Sigma_v)$ which  are induced by  an automorphism of $G_v$
has finite index in the group of all automorphisms, see \cite{DG2}.
\end{proof}

\subsection{Vertex finiteness over abelian groups}

\begin{prop}
  Assume $G$ and $\cala$ are as in Assertion 3 or 4 of Theorem \ref{main}.
Assume moreover that $G$ is one-ended.

Then vertex finiteness holds for one-edge splittings  $\Gamma$ of $G$ over $\cala$.
\end{prop}

Recall that, in Assertion 3 of Theorem \ref{main}, 
$G$ is hyperbolic relative to finitely generated abelian groups, and   $\cala$ is the family
of virtually abelian   (i.e.\ abelian or virtually cyclic) groups. In Assertion 4,  $G$ is a finitely generated, torsion-free CSA group whose abelian subgroups are finitely generated of bounded rank, and  $\cala$ is the family of abelian subgroups. 

Since Proposition \ref{cor_cv1b1a} gives vertex finiteness for one-edge virtually cyclic splittings of $G$ as   in Assertion 3,
we may assume that the edge group of $\Gamma$ is abelian. 
As in the previous subsection,  
  we assume that the choice of combinatorial attachment has been made and we distinguish several cases. We use the same notations whenever possible. 

$\bullet$  $\Lambda_v$ is a trivial 
amalgam $G_v=G_v*_{G_\eps} G_\eps$, and $G_v$ is not abelian.
For each edge $e$ of $\Gcan$ incident to $v$, the group $G_e$ is a maximal  virtually abelian subgroup of $G_v$.   
%(this is a feature of the tree of cylinders $T_c$,  equal to $T_c^*$ in the context of Assertion 3 or 4, 
%see Theorems 11.1 and 13.1 of \cite{GL3b}).
By minimality, at least one edge $e_1$  has to be attached to $u$, and $G_\varepsilon$  is conjugate to $G_{e_1}$
since $G_{e_1}\subset G_\eps^g\subset G_v$ and $G_\varepsilon$ is abelian. 
 This means that   $G_\eps$ is uniquely determined (up to conjugacy), so the marked splitting is determined (up to an inner automorphism of $G_v$). The choice of the $i_e$'s (algebraic attachment) is controlled by Remark \ref{lesi2} since 
$G_u=G_\varepsilon$  is  almost malnormal 
in $G_v$
  (as a maximal virtually abelian subgroup of a group as in Assertion 3 or 4 of Theorem \ref{main}).

$\bullet$ $\Lambda_v$ is   trivial,  
and $G_v$ is abelian.     As $G_v$ is abelian, there is no choice for the algebraic attachment.  
Finiteness will be deduced from Corollary \ref{cor_is}.

As above we denote by $u,u'$ the vertices of $\Lambda_v$
with $G_u=G_\eps$ and $G_{u'}=G_v$.
Let  $A\inc G_\varepsilon$ be the subgroup of $G_v$ generated by the  groups $G_e$ carried by edges attached to $u$ 
(this group is meaningful because  $G_v$ is   abelian; otherwise, it  changes if $G_e$ is replaced by a conjugate). Since the combinatorial attachment  is fixed, $A$ is independent of the choice of $G_\varepsilon$.

If $\eps$ does not   disconnect $\Lambda$ (\ie if $\Gamma$ is an HNN-extension), 
let $G_0$ be  the fundamental group of the graph of groups obtained from $\Lambda$ by removing the interior of 
$\varepsilon$ and changing the group carried by  $u $ from  $G_{u }=G_\varepsilon$ to $A$. 
It does  not depend   on the choice of $G_\eps$. The vertex group of $\Gamma$ is  $G_0*_A G_\eps$. By Corollary \ref{cor_is}  (applied with $P=G_v$), there are finitely many possibilities up to isomorphism.  The argument when 
 $\eps$ separates $\Lambda$ is similar.

$\bullet$ 
$\Lambda_v$ is not trivial. Then $G_v$ cannot be   rigid. The case when  it is     QH or virtually cyclic   
 is dealt with as in the previous subsection,  
so the only remaining possibility is when $G_v$ is abelian. 
 The groups carried by edges   incident to $v$ in $\Gcan$ generate a subgroup $A\inc G_v$.
 The Bass-Serre tree of $\Lambda_v$ is a line on which $G_v$ acts by translations, and $\Lambda_v$
is an HNN-extension $G_v=(G_\eps)*_{G_\eps}$  with $A\inc G_\eps\inc G_v$.
 Thus  $\Gamma$ is an HNN-extension, and its vertex group is isomorphic to $H=G_0*_A G_\eps$,
where  $G_0$ is the fundamental group of the graph of groups obtained from $\Gcan$ by changing the group carried by $v$ from  $G_v$ to $A$.  
 
The group $G_\varepsilon$    is the kernel of an epimorphism $G_v\ra \bbZ$ vanishing on $A$. 
Since there may be many such epimorphisms,  
there may be  many possibilities for the group $G_\eps\subset G_v$. 
 But Corollary \ref{cor_is} says that the isomorphism type of  $H=G_0*_A G_\varepsilon$ 
  only depends on the equivalence class of $G_\varepsilon$ as defined in Lemma \ref{sand}, 
so there are only finitely many possibilities for $H$.

\subsection{Tree finiteness  
}\label{tf}

We  
prove Theorem \ref{ttf}, i.e.\
tree finiteness for   virtually cyclic splittings $\Gamma$ of $G$ when
 $G$ is one-ended, and hyperbolic relative to virtually polycyclic groups.  We assume, of course, that $\Gamma$ has no redundant vertices.

By universal compatibility of $\Gcan$, any   splitting $\Gamma$ (possibly with several edges) may be obtained by collapsing a refinement of $\Gcan$. It therefore suffices to prove finiteness up to $\Out(G)$ for splittings $\Gamma$ which refine $\Gcan$. 

When $\Gamma$ has just one more edge than $\Gcan$, we simply apply Lemma \ref{lem_TF1} to $\Gcan$, noting that any $G_v$ which is not virtually polycyclic or QH is rigid, hence elliptic in $\Gamma$. 

In general, we pass from $\Gcan$ to $\Gamma$ by a finite sequence of one-edge refinements. Refining a QH vertex yields vertex groups which are virtually cyclic or QH, so Lemma \ref{lem_TF1} applies to each intermediate splitting. It is therefore enough to find a \emph{uniform bound (depending only on $G$) for the number of edges of $\Gamma$}  (we cannot apply \cite{BF_bounding} because $\Gamma$ does not have to be reduced in the sense of \cite{BF_bounding},  see below). We denote by $T $, $\Tcan$ the Bass-Serre trees of $\Gamma$, $\Gcan$ respectively.

We may factor the collapse map  $\pi:\Gamma\to  \Gcan$  through a   splitting $\Gamma'$, belonging to the same deformation space as $\Gamma$, such that 
the preimage of any vertex $v$ of $\Gcan$ in $\Gamma'$ is  a \emph{minimal} graph of groups:
   we obtain $\Gamma'$ from $\Gamma$ by collapsing edges of $\pi\m(v)$ associated to edges of $T $  not belonging to  the minimal subtree of  a conjugate of $G_v$ (if $G_v$ is elliptic in $\Gamma$, we collapse the whole of $\pi\m(v)$).
 
Let $T'$ be the Bass-Serre tree  of $\Gamma'$, and 
 $p:T' \ra \Tcan$    the   induced collapse map.  
We first claim that the number of edges of $\Gamma'$ is uniformly bounded. To prove this, we   consider a vertex $v$ of $\Tcan$ such that $T_v=p\m(v)$ is not a point, and we have to bound  the number of edges  of $T_v/G_v$. Note that $G_v$ is not rigid, so it is virtually polycyclic or QH, and 
the action of $G_v$ on $T_v$ is minimal. If the action of $G_v$ on $T_v$ has no redundant vertex, the number of edges of $  T_v/G_v$ is 1 if $G_v$ is polycyclic ($T_v$ is a line), bounded in terms of the orbifold $\Sigma_v$ if $G_v$ is QH. In general there may be
 redundant vertices, but such vertices have edges of $\Tcan$ attached to them, so the number of $G_v$-orbits of redundant vertices is bounded by the valence of the image of $v$ in $\Gcan$. 
This proves the claim.

Now let $q:\Gamma\ra \Gamma'$ the collapse map. We consider a vertex  $v'\in \Gamma'$, and we bound the number of edges of  $q\m(v')$. Since $\Gamma$ and $\Gamma'$ belong to the same deformation space, the group 
 $G_{v'}$ is elliptic in $\Gamma$, so 
 $q\m(v')$ is a finite tree of groups $\Lambda_{v'}$ with a vertex $w$ carrying the same group as $v'$. 
By minimality of the splitting $\Gamma$, the number of terminal vertices of $\Lambda_{v'}$ 
is bounded by the valence of $v'$ in $\Gamma'$. 
It therefore suffices to bound the length of a segment $S\inc \Lambda_{v'}$ consisting of vertices of valence 2  (these vertices make $T$ non-reduced in the sense of \cite{BF_bounding}). 

All edge stabilizers of $\Gamma$ are infinite and virtually cyclic. 
The stabilizer of any edge in $\Lambda_{v'}$ contains (with finite index) the stabilizer of an edge of $\Gamma'$. 
Moreover, if we orient $S$ towards $w$, the sequence of edge stabilizers is strictly increasing as one moves along $S$. 
Applying Assertion 3 of Lemma \ref{relpolycy} to the edge stabilizers of   $S$ then gives the required bound.

\section{One-edge splittings of arbitrary groups}
 
In this section we  prove Assertions 2, 3, 4 of Theorem \ref{main} for one-edge splittings of a group $G$ with infinitely many ends.  In all cases there is a bound for the order of finite subgroups of $G$, and $G$ is accessible.

\begin{lem}[{Compare  
\cite[Lemma 4.22]{DG2}, \cite[Theorem 18]{Wilton_one-ended}}]  
\label{dom}
Let $G$ be an accessible group with infinitely many ends. Let $C$ be a finitely generated group with finitely many ends. If $G$ splits over $C$, there is a non-trivial splitting $\Gamma'$ of $G$ over a finite group in which $C$ is elliptic.\end{lem}

\begin{proof} This is clear if $C$ has 0 or 1 end (it is elliptic in any $\Gamma'$), so assume that $C$ is virtually cyclic. Let $\Gamma$ be a non-trivial one-edge splitting of $G$ over $C$, and let $\Theta$ be a Stallings-Dunwoody decomposition of $G$  (see Subsection \ref{ts}). There are two cases. 

If $\Theta$ does not dominate $\Gamma$, some vertex group $H$ of $\Theta$ is non-elliptic in $\Gamma$, so  (up to conjugacy) splits over a subgroup $C'\inc C$. The group $C'$ is infinite because $H$ is one-ended, so it has finite index in $C$. It follows that $C$ is elliptic in $\Theta$, and we define $\Gamma'=\Theta$. 

If $\Theta$  dominates $\Gamma$, we may obtain the Bass-Serre tree of $\Gamma$ from that of $\Theta$ by   collapsing edges and performing a finite sequence of folds $T_i\to T_{i+1}$  (see \cite[p.\ 455]{BF_bounding}). There is at least one fold  because $C$ is infinite, so consider the first fold  such that $T_{i+1}$ has an infinite edge stabilizer $G_e$. There are several types of folds (see \cite{BF_bounding}), but in all cases $G_e$ is elliptic in $T_i$. As above $G_e$ has finite index in (a conjugate of) $C$, so we define $\Gamma'$ as the splitting associated to $T_i$. 
\end{proof}

\begin{rem} 
The following generalization was inspired   by N.\ Touikan.
If $G$ is as in Lemma \ref{dom}, and $\Gamma$ is a splitting of $G$ over groups with finitely many ends, there is a non-trivial splitting   of $G$ over a finite group in which all edge groups $C_i$ of $\Gamma$ are elliptic. This is proved by induction: the lemma  is true in a relative setting, and one applies  it relative to $C_1,\dots, C_i$ to the one-edge splitting of $G$ over $C_{i+1}$.
\end{rem}

 Let now $G$ be as in Theorem \ref{main}. Let $\Gamma$ be a  non-trivial one-edge splitting of $G$, say an amalgam $A*_CB$ (the argument is the same in the case of an HNN-extension). By Section \ref{feg} 
  and   the first assertion of Lemma \ref{relpolycy}, we may assume that $C$ is infinite.  It is   virtually cyclic or abelian, 
hence finitely ended.

By Lemma 3.2 of \cite{GL3a}, we can refine $\Gamma$ to a splitting $\Lambda$ which dominates the splitting $\Gamma'$ provided by Lemma \ref{dom} (but we cannot assume that $\Lambda$ collapses to $\Gamma'$). 
We may assume that all edge groups of $\Lambda$ are finite, except for the edge $e=vw$ coming from $\Gamma$ (it carries $C$). The number of edges of $\Lambda$ is bounded by Linnell's accessibility. 

All vertex groups of $\Lambda$ except $G_v$ and $G_w$ are vertex groups of a splitting of $G$ with  finite edge groups, so only finitely many isomorphism types are possible by Section \ref{feg}. Similarly, 
$H=G_v*_CG_w$ is also such a vertex group, so there are only finitely many possibilities for $H$. 
The group $H $ is elliptic in $\Gamma'$, because $G_v$ and $G_w$ are and $C$ (being infinite) fixes a unique point in the Bass-Serre tree of $\Gamma'$. Since $\Gamma'$ is non-trivial, $H$ is a proper subgroup of $G$. 

First suppose that $G$ is torsion-free. By Grushko's theorem, $H$ has rank smaller than $G$, so by induction we may assume that the theorem holds for $H$: there are finitely many possibilities for $G_v$ and $G_w$ up to isomorphism. We now see that $A$ and $B$ are fundamental groups of graphs of groups such that the number of edges is bounded, edge groups are trivial, and  only finitely many vertex groups are possible up to isomorphism. Finiteness follows.

 Now assume that $G$ has torsion (so is as in Assertion 2 or 3 of Theorem \ref{main}). There are two complications. First, one must replace  the rank by another complexity,
namely $c(H)$, defined as  the maximal number of edges in a  minimal decomposition of $H$ over finite groups without redundant vertex  (minimal means that the action on the Bass-Serre is minimal, as in Subsection \ref{ts}).
 This is finite by Linnell's accessibility, and $c(H)<c(G)$, so we can argue by induction on $c(H)$.

The groups $A$ and $B$ are now fundamental groups of graphs of groups such that the number of edges is bounded,  
and  only finitely many vertex  and edge groups are possible up to isomorphism. 
We have to control  the  
 inclusions of edge groups  into vertex groups. 
We cannot argue as in the proof of Lemma \ref{sk2} because we do not know the deformation space. 
Instead we use the fact that the vertex groups are  hyperbolic relative to virtually polycyclic groups, 
and therefore only contain finitely many conjugacy classes of finite subgroups by the first assertion of 
 Lemma \ref {relpolycy}.  

\section{Splittings with several edges} 

We now  prove Assertions 2, 3, 4 of Theorem \ref{main} in full generality, i.e.\  for   splittings with any  number of edges. Recall that we need only consider reduced splittings.

We first prove  the following claim by induction on $p$:  \emph{given $G$, there are only finitely many possible isomorphism types for vertex groups of reduced splittings $\Gamma$ of $G$ over $\cala$ with at most $p$ edges.} 

Given a vertex $v$ of $\Gamma$, choose an edge $e$ containing $v$, and collapse $e$. We get a reduced splitting $\Gamma'$ with   fewer edges, and $G_v$ is a vertex group of a one-edge splitting of  a vertex group $G_w$ of $\Gamma'$. 
This one-edge splitting is reduced  
 because $\Gamma$ is reduced, and the claim follows since by induction there are finitely many possibilities for $G_w$ up to isomorphism. 

If $G$ is finitely presented,   Bestvina-Feighn's accessibility \cite{BF_bounding} provides a bound  
for the number of edges of reduced splittings of   $G$ (note that groups in $\cala$ are small, and reduced as defined in Subsection \ref{ts} implies reduced in the sense of \cite{BF_bounding}). 
Theorem \ref{main} thus  follows from the claim when $G$ is relatively hyperbolic (Assertions 2 and 3).

Bestvina-Feighn's accessibility does not apply 
  in the CSA case   if $G$ is not finitely presented, so we use   acylindrical accessibility \cite{Sela_acylindrical,Weidmann_accessibility}  instead.  
 As usual, an abelian splitting is a splitting over abelian groups.

  \begin{lem} Let $G$ be a finitely generated, torsion-free, CSA group. Given an abelian splitting $\Gamma$ of $G$, there exists a  reduced
 2-acylindrical abelian splitting $\Gamma_c$ such that 
any non-abelian vertex group of $\Gamma$ is a  vertex group of $\Gamma_c$. 
  \end{lem}

  \begin{proof}  First assume that no edge group of $\Gamma$ is trivial. Let $T$ be the Bass-Serre tree of $\Gamma $, let $T_c$ be its  tree of cylinders  (for commutation, see Example 3.5 of  \cite{GL4}), and let $\Gamma_c=T_c/G$ be the corresponding graph of groups.  If $v$ is a vertex of $T$ with $G_v$ non-abelian, it belongs to at least two cylinders, so $G_v$ is a vertex stabilizer of $T_c$. The tree $T_c$ is 2-acylindrical (Proposition 6.3 of \cite{GL4}), and one can make it reduced by collapsing edges (this does not change non-abelian vertex stabilizers). 
  
  If certain edge groups of $\Gamma$ are trivial, perform the previous construction in each maximal subgraph consisting of edges with non-trivial group. 
\end{proof}

The claim and the lemma imply the theorem since by  acylindrical accessibility  there is a bound for the number of edges of $\Gamma_c$. We only have to control   non-abelian vertex groups of $\Gamma$ because  we assume that abelian subgroups have bounded rank.

\small

%\bibliographystyle{plain}
%\bibliography{published,unpublished}

\begin{flushleft}
Vincent Guirardel\\
Institut de Recherche Math\'ematique de Rennes\\
Membre de l'institut universitaire de France\\
Universit\'e de Rennes 1 et CNRS (UMR 6625)\\
263 avenue du G\'en\'eral Leclerc, CS 74205\\
F-35042  RENNES C\'edex\\
\emph{e-mail:} \texttt{vincent.guirardel@univ-rennes1.fr}\\[8mm]

Gilbert Levitt\\
Laboratoire de Math\'ematiques Nicolas Oresme\\
Universit\'e de Caen et CNRS (UMR 6139)\\
BP 5186\\
F-14032 Caen Cedex\\
France\\
\emph{e-mail:} \texttt{levitt@unicaen.fr}\\

\end{flushleft}

\end{document}